\documentclass[reqno, 10pt]{amsart}

\usepackage{amssymb,  latexsym}
\usepackage{hyperref}
\usepackage{amsmath}
\usepackage{amsthm}
\usepackage{enumerate}
\usepackage{amsfonts}
\usepackage{graphicx}
\usepackage{mathrsfs}
\usepackage{color}

\theoremstyle{plain}
\newtheorem{theorem}{Theorem}
\newtheorem{proposition}[theorem]{Proposition}
\newtheorem{lemma}[theorem]{Lemma}

\theoremstyle{definition}
\newtheorem{definition}[theorem]{Definition}
\newtheorem{remark}[theorem]{Remark}

\numberwithin{equation}{section}
\numberwithin{theorem}{section}

\numberwithin{equation}{section}

\newcommand{\Z}{{\mathbb{Z}}}

\newcommand{\R}{{\mathbb{R}}}
\newcommand{\T}{{\mathbb{T}}}
\newcommand{\N}{{\mathbb{N}}}

\let\Re=\undefined\DeclareMathOperator*{\Re}{Re}
\let\Im=\undefined\DeclareMathOperator*{\Im}{Im}

\def\leq{\leqslant}
\def\geq{\geqslant}

\def\le{\leqslant}
\begin{document}
\address{Mingming Deng
\newline \indent Graduate School of China Academy of Engineering Physics,  \ Beijing, \ China, \ 100089, }
\email{dengmingming18@gscaep.ac.cn}

\address{Kailong Yang
\newline \indent Chongqing National Center for Applied Mathematics, Chongqing Normal University
\newline \indent  Chongqing, China\indent }
\email{ykailong@mail.ustc.edu.cn}

\title[growth of high Sobolev norms]{On the growth of high Sobolev norms of the cubic nonlinear Schr\"odinger equation on $\R\times\T$}
%
\author{Mingming Deng and Kailong Yang}


\date{\today}

\maketitle

\begin{abstract}
	We consider the cubic nonlinear Schr\"odinger equation on product manifolds $\mathbb{R}\times \mathbb{T}$. In this paper, we obtain polynomial bounds on the growth in time of high Sobolev norms of the solutions. The main ingredient of the proof is to establish an iteration bound, which is based on the idea used by Bourgain in \cite{B1}.
\end{abstract}

\maketitle

\section{Introduction}
In this article, we study the cubic nonlinear Schr\"odinger equation on product manifolds of the form $\mathbb{R}\times \mathbb{T}$:
\begin{equation}\label{eq1.1}
\begin{cases}
i\partial_t u + \Delta_{\mathbb{R}\times \mathbb{T}}  u = |u|^2u,\\
u(0) = u_0\in H^s(\mathbb{R}\times \mathbb{T}),
\end{cases}
\end{equation}
where $\Delta_{\mathbb{R}  \times \mathbb{T}} $ is the Laplace-Beltrami operator on $\mathbb{R} \times \mathbb{T}$ and $u : \mathbb{R} \times \mathbb{R}\times \mathbb{T}  \to \mathbb{C}$ is a complex-valued function.

The equation \eqref{eq1.1} is a special case of the general nonlinear Schr\"odinger equations on the waveguides $\mathbb{R}^n \times \mathbb{T}^m$:
\begin{equation}\label{eq1.2n}
\begin{cases}
i\partial_t u + \Delta_{\mathbb{R}^n \times \mathbb{T}^m}  u = |u|^{p-1} u,\\
u(0) = u_0,
\end{cases}
\end{equation}
where $1 < p < \infty$, $m,n \in \mathbb{Z}^+$, and $\T^m$ is an m-dimensional torus. This kind of equations arise as models in the study of nonlinear optics (propagation of laser beams through the atmosphere or in a plasma), especially in nonlinear optics of telecommunications \cite{S,SL}.

 The equation \eqref{eq1.2n} has the following conserved quantities:
\begin{align}
\text{mass: }    &\quad   M(u(t))  = \int_{\mathbb{R}^n \times \mathbb{T}^m} |u(t,x)|^2\,\mathrm{d}x,\\
\text{   energy:  }     &  \quad E(u(t))  = \int_{\mathbb{R}^n \times \mathbb{T}^m} \frac12 |\nabla u(t,x)|^2  + \frac1{p+1} |u(t,x)|^{p+1} \,\mathrm{d}x.
\end{align}

Recently, there are wide range of research concerning the well-posedness theory and long time behaviors of solutions for \eqref{eq1.2n} on $\R^n \times \T^m$.
For the Euclidean case, i.e. $m=0$, \eqref{eq1.2n} has been investigated by a number of researchers, such as \cite{B1,CW,GV,MZ,S2,S3,T1}. These results are mainly based on the Strichartz inequality
\begin{equation}\label{Strichartz estimate}
\|e^{it\Delta_{\R^n}}u_0\|_{L^{\frac{2(n+2)}{n}}(\R^{n+1})} \leq C \|u_0\|_{L^2(\R^n)}.
\end{equation}
However, for the torus case $\T^m$, the exact analogue of Strichartz's inequality \eqref{Strichartz estimate} for \eqref{eq1.2n} fails since the free evolution is periodic in time. In \cite{B0,B1}, Bourgain first used the number theoretical related lattice counting arguments to prove some scale invariant Strichartz estimate in the case of periodic domains $\T^m$. After that Burq-G$\acute{e}$rard-Tzvetkov studied the general compact manifolds in \cite{BGT0,BGT,BGT1}. The reader can also consult \cite{BGT2,DGG,G1,GP,H,SPST}. Recently, there has been much interest in research of the behaviors of solutions to \eqref{eq1.2n} on $\mathbb{R}^n \times \mathbb{T}^m$. Takaoka-Tzvetkov in \cite{TT} established $L^2$ global well-posedness results for sufficiently small data on $\R\times\T$ by the $L^4-L^2$ Strichartz's inequality
$$
\|e^{it\Delta_{\R\times\T}}u_0\|_{L^4(I\times\R\times\T)} \leq C(I) \|u_0\|_{L^2(\R\times\T)}.
$$
The argument used in \cite{TT} comes from Bourgain's idea in the periodic case \cite{B0}. Cheng-Guo-Zhao in \cite{CGZ} provided global well-posedness and scattering for the defocusing quintic nonlinear Schr\"odinger equation on $\mathbb{R} \times \mathbb{T}$, and they also showed the scattering conjecture for the quintic nonlinear Schr\"odinger equation system presented by Hani-Pausader \cite{HP} on $\mathbb{R}^2 \times \mathbb{T}$. As for the general case $\mathbb{R}^n \times \mathbb{T}^m$, Barron recently in \cite{B} proved global-in-time Strichartz-type estimates based on Bourgain-Demeter \cite{BD} $l^2$ decouping method. See \cite{CGYZ,CZZ,HTT,IP,TV,ZZ} for more results in product spaces.

In this paper we apply the methods of the periodic case in \cite{B1} and show the well-posedness for \eqref{eq1.1} on $\mathbb{R} \times \mathbb{T}$ in $H^s$, $s\geq0$. One of our results is specified as follows:
\begin{theorem}[Global well-posedness]\label{GWP}
	The Cauchy problem  \eqref{eq1.1} is locally well-posed for $u_0 \in H^s(\R \times \T)$, $s > 0$, and hence globally well-posed for $u_0 \in H^s(\mathbb{R}\times \mathbb{T})$, $s \geq 1$.
\end{theorem}

\begin{remark}
Takaoka and Tzvetkov obtained the $L^2$ well-posedness results for \eqref{eq1.1} with small data in \cite{TT}. Therefore, we only consider the well-posedness for $u_0 \in H^s(\R \times \T)$, $s > 0$ in the rest of paper. For the proof of Theorem \ref{GWP}, we will apply the fixed point argument for the equivalent integral equation \eqref{eq1.1} in Bourgain space $X^{s,b}([0,\delta])$ with $b>\frac12$, $s>0$ and $\delta>0$ small enough. The essential point in what follows is the nonlinear estimate(trilinear estimate)
\begin{equation}\label{eq-1.12}
\|u_1\bar{u}_2u_3\|_{X_\delta^{s,b'-1}(\R\times\T)} \leq C \|u_1\|_{X_\delta^{s,b}(\R\times\T)} \|u_2\|_{X_\delta^{s,b}(\R\times\T)} \|u_3\|_{X_\delta^{s,b}(\R\times\T)},
\end{equation}
where $s>0$ and $b'>b>\frac12$. We derive it in Section 3 from Fourier multiplier method in \cite{B1} and a localized versions of the Strichartz inequality on $\R\times\T$
\begin{equation}\label{1.11}
\|e^{it\Delta_{\R\times\T}}P_{\leq N}u\|_{L^4(I\times\R\times\T)} \lesssim N^\varepsilon \|u\|_{L^2(\R\times\T)},
\end{equation}
which is established by Barron \cite{B} based on the $l^2$ decoupling method. This yields local well-posedness in the corresponding space. Then making a further discussion of the size of $\delta$, we prove the global well-posedness for $u_0 \in H^s(\mathbb{R}\times \mathbb{T})$, $s \geq 1$.
\end{remark}

Theorem \ref{GWP} implies that for $u_0\in H^s(\R\times\T)$ with $s\geq1$, there exists a unique global solution $u\in C(\R,H^s(\R\times\T))$ to \eqref{eq1.1}. Using the conservations of mass and energy, one can see that the $H^1$-norm of the solution is controlled by some constants. It is natural to ask what will happens to the $H^s$-norm with $s>1$, when $|t| \gg 1$?

Certainly, the problem has physical significance since it quantifies the Low-to-High frequency cascade, i.e. how much of the support of $|\hat{u}|^2$ has shifted from the low to high frequencies. Therefore, the growth of high Sobolev norms gives a quantitative estimate on the Low-to-High frequency cascade. The phenomenon of such a cascade in a dispersive wave model has been studied since 1960s and see for instance \cite{BN,H',Z}.

Suppose that $u$ is a solution to \eqref{eq1.1}, then by iterating the local well-posedness scheme in \cite{B0,B1,T1}, one can immediately obtain exponential bound on the growth of Sobolev norms. Namely, there exists $T=T(\|u_0\|_{H^s})$, such that
	\begin{equation}
		\|u(t)\|_{H^s}\leq C\|u(\tau)\|_{H^s},
		\label{ex}
	\end{equation}
whenever $t\in[\tau, \tau+T]$.
	Then using Picard's iteration, one infers that
	\begin{equation}
		\|u(t)\|_{H^s}\leq Ce^{|t|}.
		\label{ex1}
	\end{equation}
Later, Bourgain \cite{b1} first observed that \eqref{ex} can be improved to
	\begin{equation}
		\|u(t)\|_{H^s}\leq \|u(\tau)\|_{H^s}+C\|u(\tau)\|_{H^s}^{1-\delta},\qquad\delta^{-1}=(s-1)+
		\label{main2}
	\end{equation}
on $\mathbb{T}^2$ for Schr\"odinger equation by the Fourier multiplier method, thus \eqref{ex1} can be refined to
	\begin{equation}
	\|u(t)\|_{H^s}\leq C|t|^{2(s-1){+}}.
	\label{main3}
	\end{equation}
    After that Staffilani \cite{s,s1} showed that the solutions to some-type of KdV and Schr\"odinger equations on $\R$ and $\R^2$ possesses the polynomial bound on the time growth of $H^s$-norm, $s>1$ by using fine multilinear estimates. Colliander-Keel-Staffilani-Takaoka-Tao established polynomial bounds in low Sobolev norm with $s\in(0,1)$ for the NLS equation by proposing a new method using modified energy called the ``upside-down I-method" in \cite{CKSTT1}. Then Sohinger in \cite{SO,SO1} developed the upside-down I-method to obtain polynomial bounds on the growth of high Sobolev norms for NLS on the circle $\mathbb{S}$ and $\R$. In addition, we also refer to \cite{CWa,d1,ptv,Z1} and the references therein for further developments in this topic.

    The main result of this paper is demonstrated as following:
\begin{theorem}\label{growth}
Assume $u_0 \in H^s(\R\times\T)$, $s>1$. Then the global solution $u$ for \eqref{eq1.1} satisfies
\begin{equation}
\|u(t)\|_{H^s} \leq C|t|^{2(s-1)+}, \quad \text{for} ~|t| \gg 1,
\end{equation}
here the constant $C$ is dependent on $M(u_0)$ and $E(u_0)$.
\end{theorem}

The idea of the proof for Theorem \ref{growth} is inspired by Bourgain's work in \cite{B1}. Comparing to $\R^2$ case, we can not expect the dispersive effects on $\R\times\T$, so the natural problems and methods are different. We only try to connect with the case of periodic $\T^2$ from \cite{B1}.


In order to obtain the polynomial bound in Theorem \ref{growth}, it suffices to prove for $T=T(\|u_0\|_{H^1})$
\begin{equation}\label{eq-1.11}
		\|u(t)\|_{H^s}\leq \|u(\tau)\|_{H^s}+C\|u(\tau)\|_{H^s}^{1-\gamma}.
	\end{equation}
for any $t\in[\tau, \tau+T]$ and some $\gamma=\gamma(s)$. Indeed, from the special structure of \eqref{eq1.1}, \eqref{eq-1.11} can be reduced to estimate the nonlinear term of \eqref{eq1.1}. Then we can obtain the iteration bound \eqref{eq-1.11} by the upper bound nonlinear estimate \eqref{eq-1.12}. Once this iteration bound is established, a classical reduction argument leads to the polynomial growth.

Finally the rest of the paper is organized as follows: Section 2 provides some notations and known results, which will be used hereinafter. In Section 3, we will prove a nonlinear estimate in Bourgain space $X^{s,b}(\R\times\R\times\T)$, and give a proof for Theorem \ref{GWP}. Finally we prove Theorem \ref{growth} in Section 4.

\section{preliminaries}
As is standard, we use $A\lesssim B$ to denote the statement that $A\le CB$ for some large harmless constant C which may vary from line to line. We use $A\sim B$ to denote the statement that $A\lesssim B\lesssim A$. And we use $a\pm$ to denote $a\pm\varepsilon$ for any $\varepsilon\ll1$. We use the standard Lebesgue norms
	$$
	\|u\|_{L_x^p(\R\times\T)}=\left( \int_{\R\times\T}|u(x)|^p{\rm d}x\right)^{\frac1p} ,\qquad\|u\|_{L^q_tL^p_{x}(\R\times\R\times\T)}=\left\| \|u(x)\|_{L^p_x(\R\times\T)}\right\| _{L^q_t(\R)},
	$$
	and the Fourier transform on $\R\times\T$
	$$
	\hat{u}(\xi,n)=\int_{\R\times\T}e^{-2\pi i(x_1,x_2)\cdot (\xi,n)}u(x){\rm d}x, \quad x=(x_1,x_2)\in\R\times\T.
	$$
And we also note the Fourier inversion transform
$$
u(x) = \sum_{n\in\Z}\int_{\xi\in\R}e^{2\pi i(\xi,n)\cdot (x_1,x_2)}\hat{u}(\xi,n){\rm d}\xi, \quad x=(x_1,x_2)\in\R\times\T.
$$
\begin{definition}\label{bourgain space}
Let $X^{s,b}(\R\times\R\times\T)$ be the Bourgain space associated to the 2-dimensional Schr\"odinger equation with data on $\R\times\T$, equipped with the norm
\begin{equation}
\|u\|_{X^{s,b}(\R\times\R\times\T)} = \Big( \sum_{n\in\Z}\int_{\R}\int_{\R} (1 + |\tau - \xi^2 - n^2|)^{2b} (1 + |\xi| + |n|)^{2s} |\hat{u}(\tau,\xi,n)|^2 {\rm d}\tau {\rm d}\xi \Big)^{\frac12}.
\end{equation}
\end{definition}

Then for $0 < T \leq 1$, we denote by $X^{s,b}_T(\R\times\T)$ the space of elements of $X^{s,b}(\R\times\R\times\T)$ endowed with the norm
\begin{equation}
\|u\|_{X^{s,b}_T(\R\times\T)} = \inf\big\{ \|\tilde{u}\|_{X^{s,b}(\R\times\R\times\T)},~ \tilde{u}|_{(-T,T)\times\R\times\T} = u \big\}.
\end{equation}

We start with a brief review of some basic properties of Bourgain space.
\begin{proposition}\label{Prop-2.2}\
\begin{enumerate}
  \item $u \in X^{s,b}(\R\times\R\times\T) \Longleftrightarrow e^{it\Delta}u(t,\cdot) \in H^b(\R; H^s(\R\times\T))$.
  \item For $b>\frac12$, $X^{s,b}(\R\times\R\times\T) \hookrightarrow C(\R; H^s(\R\times\T))$, and $X^{s,b}_T(\R\times\T) \hookrightarrow C((-T,T); H^s(\R\times\T))$.
  \item For $s_1 \leq s_2, b_1 \leq b_2$, $X^{s_2,b_2}(\R\times\R\times\T) \hookrightarrow X^{s_1,b_1}(\R\times\R\times\T)$.
\end{enumerate}
\end{proposition}
%
%

Next we give a bound for $\|\chi_{[c,d]}(t)u\|_{X^{s,b}(\R\times\R\times\T)}$. A similar result was proved in \cite{CKSTT} and \cite{SO}, but in different spaces.
\begin{lemma}\label{lemma-t}
Let $c,d\in\R$ such that $c<d$. If $b\in(0,\frac12)$, and $s\in\R$, then one has
\begin{equation*}
\|\chi_{[c,d]}(t)u\|_{X^{s,b}(\R\times\R\times\T)}\lesssim\|u\|_{X^{s,b+}(\R\times\R\times\T)}
\end{equation*}
where the implicit constant does not depend on $c,d$.
\end{lemma}

Now we define the Littlewood-Paley decomposition. For a dyadic integer $1\leq N = 2^j, j\in\Z$. We let
$$
\mathcal{F}(P_{\leq N}u)(\xi,n) = \psi(\xi/N)\psi(n/N)\hat{u}(\xi,n),
$$
where $\psi\in C_0^\infty(\R)$ is a smooth bump function such that
\begin{equation*}
\psi(\xi)=
\begin{cases}
1,&~ |\xi|\leq 1,\\
0,&~ |\xi|\geq 2.
\end{cases}
\end{equation*}
And we revisit a local Strichartz estimate from Proposition 3.4 in Barron \cite{B}.
\begin{lemma}\label{proposition-2.3}
For any bounded time interval $I$. One has
\begin{equation}\label{eq-2.3}
\|e^{it\Delta_{\R\times\T}}P_{\leq N}u\|_{L^4(I\times\R\times\T)} \lesssim N^\varepsilon \|u\|_{L^2(\R\times\T)}.
\end{equation}
\end{lemma}

Choosing a ball $Q$ in $\R\times\Z$ of size $N$ entered at a point, and making a change of variable to $(\xi,n)$, we can rewrite \eqref{eq-2.3} as
\begin{equation}\label{eq-2.3'}
\|e^{it\Delta_{\R\times\T}}u\|_{L^4(I\times\R\times\T)} \lesssim N^\varepsilon \|u\|_{L^2(\R\times\T)},
\end{equation}
where $supp ~\hat{u}\in Q$.

Then from Burq-G$\acute{e}$rard-Tzvetkov \cite{BGT}, \eqref{eq-2.3'} is equivalent to the following lemma.
\begin{lemma}\label{Lemma2.4}
For any $b_1>\frac12$ and $u,v \in X^{0,b_1}(\R\times\R\times\T)$ satisfying $supp ~\hat{u}\in Q_1,supp ~\hat{v}\in Q_2$, where the ball $Q_j$ of size $N_j,j=1,2$. One has
\begin{equation}\label{eq-2.4}
\|uv\|_{L^2(\R\times\R\times\T)} \leq C (N_1N_2)^\varepsilon \|u\|_{X^{0,b_1}(\R\times\R\times\T)}\|v\|_{X^{0,b_1}(\R\times\R\times\T)}.
\end{equation}
\end{lemma}
\begin{proof}
Suppose that $u(t)$, and $v(t)$ are supported in time interval $(0,1)$ and write
$$
u(t) = e^{-it\Delta_{\R\times\T}}e^{it\Delta_{\R\times\T}}u(t) := e^{-it\Delta_{\R\times\T}}U(t),
$$
and
$$
v(t) = e^{-it\Delta_{\R\times\T}}e^{it\Delta_{\R\times\T}}v(t) := e^{-it\Delta_{\R\times\T}}V(t).
$$
Then we have
$$
u(t) = \frac{1}{2\pi}\int_{\R}e^{it\tau}e^{-it\Delta_{\R\times\T}}\hat{U}(\tau){\rm d}\tau, \quad v(t) = \frac{1}{2\pi}\int_{\R}e^{it\tau}e^{-it\Delta_{\R\times\T}}\hat{V}(\tau){\rm d}\tau
$$
and
$$
(uv)(t) = \frac{1}{(2\pi)^2}\int_{\R}\int_{\R}e^{it(\tau_1+\tau_2)}e^{-it\Delta_{\R\times\T}}\hat{U}(\tau_1)
e^{-it\Delta_{\R\times\T}}\hat{V}(\tau_2){\rm d}\tau_1{\rm d}\tau_2.
$$
From \eqref{eq-2.3'} and the Cauchy-Schwarz inequality in $(\tau_1,\tau_2)$, one can obtain
\begin{equation}
\begin{aligned}
\|uv\|_{L^2((0,1)\times\R\times\T)} \leq & C(N_1N_2)^\varepsilon\int_{\tau_1,\tau_2}\|\hat{U}(\tau_1)\|_{L^2(\R\times\T)}\|\hat{V}(\tau_2)\|_{L^2(\R\times\T)}{\rm d}\tau_1{\rm d}\tau_2\\
\leq & C(N_1N_2)^\varepsilon\|(1+|\tau_1-\xi_1^2-n_1^2|)^b\hat{U}(\tau_1)\|_{L^2(\R\times\T)}\\
&\times \|(1+|\tau_2-\xi_2^2-n_2^2|)^b\hat{V}(\tau_2)\|_{L^2(\R\times\T)}\\
\leq & C(N_1N_2)^\varepsilon\|u\|_{X^{0,b_1}(\R\times\T)}\|v\|_{X^{0,b_1}(\R\times\T)},
\end{aligned}
\end{equation}
for $b_1>\frac12$. Ultimately, decomposing $\R_t = \Z + (0,1)$ and $u(t)=\sum_{m\in\Z}\psi(t-\frac{m}2)u(t)$ and $v(t)=\sum_{m\in\Z}\psi(t-\frac{m}2)v(t)$  with a suitable $\psi \in C_0^\infty(\R)$ supported in $(0,1)$, the general case for $u(t)$ and $v(t)$ follows from the considered particular case of $u(t)$ and $v(t)$ supported in times in the interval $(0,1)$.
\end{proof}

Furthermore, we get the following additional local estimate by simple calculations.
\begin{lemma}\label{l4' esitimate}
For $b_2 > \frac14$. Assume $supp ~\hat{u}\in Q$ with the ball $Q$ of size $N$, then the following inequality holds:
\begin{equation}
\|u\|_{L^4(\R\times\R\times\T)} \leq C N^{\frac14}\|u\|_{X^{0,b_2}(\R\times\R\times\T)}.
\end{equation}
\end{lemma}
\begin{proof}
From the Hausdorff - Young and H\"older inequalities, we obtain
$$
\begin{aligned}
\|u\|_{L^4(\R\times\R\times\T)} \lesssim & \Big( \sum_{n\in Q}\int_{\xi\in Q}\int|\hat{u}(\tau,\xi,n)|^{\frac43}{\rm d}\tau{\rm d}\xi \Big)^{\frac34}\\
\lesssim & \Big( \sum_{n\in Q}\int_{\xi\in Q}\big(\int(1+|\tau-\xi^2-n^2|)^{2b_2}|\hat{u}(\tau,\xi,n)|^{2}{\rm d}\tau\big)^{\frac23}{\rm d}\xi \Big)^{\frac34}\\
\lesssim & N^\frac14 \Big( \sum_{n\in Q}\int_{\xi\in Q}\int(1+|\tau-\xi^2-n^2|)^{2b_2}|\hat{u}(\tau,\xi,n)|^{2}{\rm d}\tau{\rm d}\xi \Big)^{\frac12}
\end{aligned}
$$
for $b_2 > \frac14$.
\end{proof}

Then interpolation between Lemma \ref{Lemma2.4} and Lemma \ref{l4' esitimate} yields the following lemma.
\begin{lemma}\label{lemma-main}
For any $0 < s_0 < \frac14$. Assume $supp ~\hat{u}\in Q$ with the ball $Q$ of size $N$, there exists some $b_3\in(\frac{1-2s_0}{2},\frac12)$ satisfying
\begin{equation}
\|u\|_{L^4(\R\times\R\times\T)} \leq C N^{s_0}\|u\|_{X^{0,b_3}(\R\times\R\times\T)}.
\end{equation}
\end{lemma}

\section{global~well-posedness}
The idea of the proof of Theorem \ref{GWP} is inspired by Bourgain in \cite{B1}. Before proving this, let us first state the nonlinear estimate in $X^{s,b}$.

\subsection{The~nonlinear~estimate~in~$X^{s,b}$}

\begin{theorem}\label{theorem-nonlinear estimate}
Let $s>0$. There exists $\frac12 < b \leq b'$ and $C > 0$ such that for $u_j, j=1,2,3$ in $X^{s,b}(\R\times\R\times\T)$ satisfying
\begin{equation}\label{nonlinear estimate}
\|u_1\bar{u}_2u_3\|_{X^{s,b'-1}(\R\times\R\times\T)} \leq C \|u_1\|_{X^{s,b}(\R\times\R\times\T)} \|u_2\|_{X^{s,b}(\R\times\R\times\T)} \|u_3\|_{X^{s,b}(\R\times\R\times\T)}.
\end{equation}
\end{theorem}
\begin{proof}
By duality, we have
\begin{align}\label{eq-3.2}
& ~\|u_1\bar{u}_2u_3\|_{X^{s,b'-1}(\R\times\R\times\T)}\nonumber\\
= & ~\Big( \sum_n\int\int (1+|\tau-\xi^2-n^2|)^{2(b'-1)}(1+|\xi|+|n|)^{2s}\big| \widehat{(u_1\bar{u}_2u_3)}(\tau,\xi,n) \big|^2{\rm d}\tau{\rm d}\xi \Big)^{\frac12}\nonumber\\
= & ~\sup_{\|\hat{v}\|_{L^2}=1}\sum_n\int\int\sum_{n_1}\int\int\sum_{n_2}
\int\int(1+|\tau-\xi^2-n^2|)^{(b'-1)}(1+|\xi|+|n|)^{s}\nonumber\\
& ~\times \hat{u}_1(\tau_1,\xi_1,n_1)\bar{\hat{u}}_2(-\tau_2,-\xi_2,-n_2)
\hat{u}_3(\tau-\tau_1-\tau_2,\xi-\xi_1-\xi_2,n-n_1-n_2)\nonumber\\
& \qquad\qquad\qquad~\times \bar{\hat{v}}(\tau,\xi,n){\rm d}\tau_1{\rm d}\xi_1{\rm d}\tau_2{\rm d}\xi_2{\rm d}\tau{\rm d}\xi\nonumber\\
\leq & ~\sup_{\|\hat{v}\|_{L^2}=1}\sum\int\int\frac{|\bar{\hat{v}}(\tau,\xi,n)|}{(1+|\tau-\xi^2-n^2|)^{(1-b')}}(1+|\xi|+|n|)^{s}|\hat{u}_1(\tau_1,\xi_1,n_1)|
\nonumber\\
& \qquad\qquad\qquad ~\times |\bar{\hat{u}}_2(-\tau_2,-\xi_2,-n_2)||\hat{u}_3(\tau_3,\xi_3,n_3)|{\rm d}\tau_1{\rm d}\xi_1{\rm d}\tau_2{\rm d}\xi_2{\rm d}\tau_3{\rm d}\xi_3,
\end{align}
where $n=n_1-n_2+n_3$, $\xi=\xi_1-\xi_2+\xi_3$, and $\tau=\tau_1-\tau_2+\tau_3$.

Observe that
\begin{equation*}
|\xi|+|n| \leq C\max\{|\xi_1|+|n_1|,|\xi_2|+|n_2|,|\xi_3|+|n_3|\},
\end{equation*}
and
\begin{equation*}
(1+|\xi|+|n|)^{s} \leq C\max\{(1+|\xi_1|+|n_1|)^{s},(1+|\xi_2|+|n_2|)^{s},(1+|\xi_3|+|n_3|)^{s}\}.
\end{equation*}
Due to symmetry, we can assume $(1+|\xi|+|n|)^{s} \lesssim (1+|\xi_1|+|n_1|)^{s}$. Then \eqref{eq-3.2} is bounded by
\begin{equation}
\begin{aligned}\label{eq-3.3}
&\sup_{\|\hat{v}\|_{L^2}=1}\sum_{n_1,n_2,n_3}\int\int\frac{|\bar{\hat{v}}(\tau,\xi,n)|}{(1+|\tau-\xi^2-n^2|)^{(1-b')}}(1+|\xi_1|+|n_1|)^{s}\\
~\times &|\hat{u}_1(\tau_1,\xi_1,n_1)||\bar{\hat{u}}_2(-\tau_2,-\xi_2,-n_2)||\hat{u}_3(\tau_3,\xi_3,n_3)|{\rm d}\tau_1{\rm d}\xi_1{\rm d}\tau_2{\rm d}\xi_2{\rm d}\tau_3{\rm d}\xi_3.
\end{aligned}
\end{equation}

For each $(\xi_i,n_i)$-index $i=1,2,3$, by a standard dyadic partitioning as in Littlewood-paley theory, the $\R\times\Z$-index set can be divided into dyadic regions,
\begin{equation*}
\R\times\Z = \bigcup_{k\in\N}D_k, \quad \text{and} \quad D_k = \big\{ (\xi,n)\in\R\times\Z: (1+|\xi|+|n|)\sim2^k \big\}.
\end{equation*}
Write
\begin{equation*}
\sum\int_{(\xi_i,n_i)\in\R\times\Z} = \sum_{k_1\geq k_2\geq k_3}\sum\int_{(\xi_i,n_i)\in D_{k_i}} \quad \text{for}~ i=1,2,3.
\end{equation*}
For $k_1\geq k_2 \geq k_3$, we make a further of $D_{k_1} = \bigcup_{\alpha} Q_{\alpha}$ in ball of size $2^{k_2}$. One may write
\begin{equation}
\begin{aligned}\label{eq-3.4}
& 2^{k_1s}\sum\int_{(\xi_i,n_i)\in D_{k_i}}\int\frac{|\bar{\hat{v}}(\tau,\xi,n)|}{(1+|\tau-\xi^2-n^2|)^{(1-b')}}|\hat{u}_1(\tau_1,\xi_1,n_1)||\bar{\hat{u}}_2(-\tau_2,-\xi_2,-n_2)|\\
& \qquad\qquad\qquad \times|\hat{u}_3(\tau_3,\xi_3,n_3)|{\rm d}\tau_1{\rm d}\xi_1{\rm d}\tau_2{\rm d}\xi_2{\rm d}\tau_3{\rm d}\xi_3 \sim\\
& 2^{k_1s}\sum_\alpha \sum \int_{(\xi,n),(\xi_1,n_1)\in Q_\alpha, (\xi_2,n_2)\in D_{k_2}, (\xi_3,n_3)\in D_{k_3}}
\int\frac{|\bar{\hat{v}}(\tau,\xi,n)|}{(1+|\tau-\xi^2-n^2|)^{(1-b')}}\\
& \qquad \times|\hat{u}_1(\tau_1,\xi_1,n_1)||\bar{\hat{u}}_2(-\tau_2,-\xi_2,-n_2)||\hat{u}_3(\tau_3,\xi_3,n_3)|{\rm d}\tau_1{\rm d}\xi_1{\rm d}\tau_2{\rm d}\xi_2{\rm d}\tau_3{\rm d}\xi_3.
\end{aligned}
\end{equation}

Next we recall a fact from Fourier analysis. For simplicity, we suppose that $f_1,f_2,f_3,f_4$ are functions on $\R^2$ and all $\widehat{f_j}$ with $j=1,2,3,4$ are real-valued. One has
\begin{equation}\label{eq-3.5'}
\int f_1\bar{f_2}f_3\bar{f_4}{\rm d}x = \int_{\xi_1-\xi_2+\xi_3-\xi_4=0}\widehat{f_1}(\xi_1)\widehat{\bar{f_2}}(\xi_2)\widehat{f_3}(\xi_3)
\widehat{\bar{f_4}}(\xi_4){\rm d}\xi_1{\rm d}\xi_2{\rm d}\xi_3{\rm d}\xi_4.
\end{equation}

Using the analogue of \eqref{eq-3.5'} for the space-time Fourier transform on $\R\times\R\times\T$, thus we collect
\begin{equation}\label{eq-3.5}
\eqref{eq-3.4} \leq \int\int F_\alpha(x,t) G_\alpha(x,t) H_2(x,t) H_3(x,t) {\rm d}x{\rm d}t,
\end{equation}
where
$$
\begin{aligned}
F_\alpha(x,t) = & \sum\int_{(\xi,n)\in Q_\alpha}\int\frac{|\bar{\hat{v}}(\tau,\xi,n)|}{(1+|\tau-\xi^2-n^2|)^{(1-b')}}e^{i((\xi,n)x+\tau t)}{\rm d}\tau{\rm d}\xi,\\
G_\alpha(x,t) = & \sum\int_{(\xi,n)\in Q_\alpha}\int|\hat{u}_1(\tau,\xi,n)|e^{i((\xi,n)x+\tau t)}{\rm d}\tau{\rm d}\xi,\\
H_2(x,t) = & \sum\int_{(\xi,n)\in D_{k_2}}\int|\hat{\bar{u}}_2(\tau,\xi,n)|e^{i((\xi,n)x+\tau t)}{\rm d}\tau{\rm d}\xi,\\
H_3(x,t) = & \sum\int_{(\xi,n)\in D_{k_3}}\int|\hat{u}_3(\tau,\xi,n)|e^{i((\xi,n)x+\tau t)}{\rm d}\tau{\rm d}\xi.
\end{aligned}
$$
Thus \eqref{eq-3.5} is bounded by
\begin{equation}\label{eq-3.6}
\eqref{eq-3.5} \leq C\|F_\alpha\|_{L^4}\|G_\alpha\|_{L^4}\|H_2\|_{L^4}\|H_3\|_{L^4}.
\end{equation}
Choose $0<s_0<\frac14$ and $\frac{2-s_0}{2}<b_3<\frac12$ satisfing Lemma \ref{lemma-main}. Then
$$
\begin{aligned}
\|G_\alpha\|_{L^4(\R\times\R\times\T)} \leq & C2^{k_2s_0}\Big( \sum\int_{(\xi,n)\in Q_\alpha}\int(1+|\tau-\xi^2-n^2|)^{2b_3}|\hat{u}_1(\tau,\xi,n)|^2{\rm d}\tau{\rm d}\xi \Big)^\frac12,\\
\|H_i\|_{L^4(\R\times\R\times\T)} \leq & C2^{k_is_0}\Big( \sum\int_{(\xi,n)\in D_{k_i}}\int(1+|\tau-\xi^2-n^2|)^{2b_3}|\hat{u}_i(\tau,\xi,n)|^2{\rm d}\tau{\rm d}\xi \Big)^\frac12
\end{aligned}
$$
for $i=2,3$. And
\begin{equation}\label{eq-3.8'}
\|F_\alpha\|_{L^4(\R\times\R\times\T)} \leq C2^{k_2s_0}\Big( \sum\int_{(\xi,n)\in Q_\alpha}\int(1+|\tau-\xi^2-n^2|)^{2b_3}|\widehat{F}(\tau,\xi,n)|^2{\rm d}\tau{\rm d}\xi \Big)^\frac12.
\end{equation}

Consequently, summing over $\alpha$, and assuming $1-b'>b_3$, we have
\begin{equation}
\begin{aligned}\label{eq-3.7}
& \eqref{eq-3.6} \leq  C2^{2k_2s_0}\|v|_{D_{k_1}}\|_{L^2}\\
& \times \Big( \sum\int_{(\xi,n)\in D_{k_1}}\int(1+|\tau-\xi^2-n^2|)^{2b_3}(1+|\xi|+|n|)^{2s}|\hat{u}_1(\tau,\xi,n)|^2{\rm d}\tau{\rm d}\xi \Big)^\frac12\\
& \times \Big( \sum\int_{(\xi,n)\in D_{k_2}}\int(1+|\tau-\xi^2-n^2|)^{2b_3}(1+|\xi|+|n|)^{2s_0}|\hat{u}_2(\tau,\xi,n)|^2{\rm d}\tau{\rm d}\xi \Big)^\frac12\\
& \times \Big( \sum\int_{(\xi,n)\in D_{k_3}}\int(1+|\tau-\xi^2-n^2|)^{2b_3}(1+|\xi|+|n|)^{2s_0}|\hat{u}_3(\tau,\xi,n)|^2{\rm d}\tau{\rm d}\xi \Big)^\frac12.
\end{aligned}
\end{equation}
Note that $k_1 \geq k_2 \geq k_3$, which yields
$$
\begin{aligned}
& \eqref{eq-3.7} \leq C\|v_{D_{k_1}}\|_{L^2}\\
& \times \Big( \sum\int_{(\xi,n)\in D_{k_1}}\int(1+|\tau-\xi^2-n^2|)^{2b_3}(1+|\xi|+|n|)^{2s}|\hat{u}_1(\tau,\xi,n)|^2{\rm d}\tau{\rm d}\xi \Big)^\frac12\\
& \times \Big( \sum\int_{(\xi,n)\in D_{k_2}}\int(1+|\tau-\xi^2-n^2|)^{2b_3}(1+|\xi|+|n|)^{6s_0}|\hat{u}_2(\tau,\xi,n)|^2{\rm d}\tau{\rm d}\xi \Big)^\frac12\\
& \times \Big( \sum\int_{(\xi,n)\in D_{k_3}}\int(1+|\tau-\xi^2-n^2|)^{2b_3}(1+|\xi|+|n|)^{2s_0}|\hat{u}_3(\tau,\xi,n)|^2{\rm d}\tau{\rm d}\xi \Big)^\frac12.
\end{aligned}
$$
Here, summation on $k_1 \geq k_2 \geq k_3$ permits to bound \eqref{eq-3.2} by
\begin{equation}
\begin{aligned}\label{eq-3.8}
\eqref{eq-3.2} \leq & C\Big( \sum_n\int\int(1+|\tau-\xi^2-n^2|)^{2b_3}(1+|\xi|+|n|)^{2s}|\hat{u}(\tau,\xi,n)|^2{\rm d}\tau{\rm d}\xi \Big)^\frac12\\
& \times \Big( \sum_n\int\int(1+|\tau-\xi^2-n^2|)^{2b_3}(1+|\xi|+|n|)^{8s_0}|\hat{u}_2(\tau,\xi,n)|^2{\rm d}\tau{\rm d}\xi \Big)^\frac12,\\
& \times \Big( \sum_n\int\int(1+|\tau-\xi^2-n^2|)^{2b_3}(1+|\xi|+|n|)^{8s_0}|\hat{u}_3(\tau,\xi,n)|^2{\rm d}\tau{\rm d}\xi \Big)^\frac12.
\end{aligned}
\end{equation}
Fixing $s>0$ and letting $s_0 = \min{\{\frac{s}{4},\frac14\}}$, we can take \textcolor[rgb]{1.00,0.00,0.00}{$b=b' = b(s)>\frac12$} such that $1-b(s)>\frac{1-2s_0}{2}$ with $b_3\in(\frac{1-2s_0}{2},\frac12)$. Therefore, from \eqref{eq-3.8}, we get
\begin{equation}\label{eq-3.9}
\begin{aligned}
\|u_1\bar{u}_2u_3\|_{X^{s,b'-1}(\R\times\R\times\T)} \leq & ~C \|u_1\|_{X^{s,b_3}(\R\times\R\times\T)} \|u_2\|_{X^{s_0,b_3}(\R\times\R\times\T)} \|u_3\|_{X^{s_0,b_3}(\R\times\R\times\T)}\\
\leq & ~C \|u_1\|_{X^{s,b}(\R\times\R\times\T)} \|u_2\|_{X^{s,b}(\R\times\R\times\T)} \|u_3\|_{X^{s,b}(\R\times\R\times\T)}.
\end{aligned}
\end{equation}
Thus the proof is completed.
\end{proof}

\subsection{the~proof~of~theorem~\ref{GWP}}
In fact, Takaoka and Tzvetkov in \cite{TT} determined the $L^2$ global well-posedness results for \eqref{eq1.1} with small data. In this subsection, we shall state the well-posedness for $u_0\in H^s,s>0$ briefly.

First, we apply a fixed point argument for the integral equation corresponding to equation \eqref{eq1.1}
\begin{equation}\label{eq-duhamel}
u(t) = e^{it\Delta_{\R\times\T}}u_0 - i\int_0^te^{i(t-s)\Delta_{\R\times\T}}(|u|^2u)(s){\rm d}s.
\end{equation}
Recall that $\phi\in C_0^\infty(\R)$ is a cut-off function such that $supp~\phi \subset(-2,2)$, $\phi\equiv1$ on the interval $[-1,1]$, and $\phi_\delta(t) = \phi(t/\delta)$ for $\delta>0$. Consider a truncated version of \eqref{eq-duhamel}
\begin{equation}\label{eq-truncated duhamel}
u(t) = \phi(t)e^{it\Delta_{\R\times\T}}u_0 - i\phi_\delta(t)\int_0^te^{i(t-s)\Delta_{\R\times\T}}(|u|^2u)(s){\rm d}s.
\end{equation}
We solve \eqref{eq-truncated duhamel} by the fixed point argument in $X^{s,b}_\delta(\R\times\T)$ for suitable $s>0, b>\frac12$ and sufficiently small $\delta>0$. We first state some estimates from Lemma 3.1 and Lemma 3.2 in \cite{G}.
\begin{lemma}\label{lemma-ginibre}
For $b'\geq b > \frac12$, and $0<\delta\leq1$. Then it holds
$$
\begin{aligned}
\|\phi(t)e^{it\Delta_{\R\times\T}}u_0\|_{X^{s,b}(\R\times\R\times\T)} \leq & C\|u_0\|_{H^s(\R\times\T)},\\
\big\| \phi_\delta(t)\int_0^te^{i(t-s)\Delta_{\R\times\T}}(|u|^2u)(s){\rm d}s \big\|_{X^{s,b}(\R\times\R\times\T)} \leq & C \delta^{b'-b}\||u|^2u\|_{X^{s,b'-1}(\R\times\R\times\T)},
\end{aligned}
$$
provided $s\in\R$ and some constants $C>0$.
\end{lemma}

Define an operator $\Phi$ as
\begin{equation}\label{eq-3.14}
\Phi(u(t)) := \phi(t)e^{it\Delta_{\R\times\T}}u_0 - i\phi_\delta(t)\int_0^te^{i(t-s)\Delta_{\R\times\T}}(|u|^2u)(s){\rm d}s.
\end{equation}
Lemma \ref{lemma-ginibre}, Theorem \ref{theorem-nonlinear estimate} and Lemma \ref{lemma-t} yield
\begin{equation}
\|\Phi(u)\|_{X^{s,b}_\delta(\R\times\T)} \leq C\|u_0\|_{H^s(\R\times\T)} + C\delta^{b'-b}\|u\|_{X^{s,b}_\delta(\R\times\T)}^3
\end{equation}
for $b'\geq b >\frac12$. Similarly
\begin{equation}
\|\Phi(u) - \Phi(v)\|_{X^{s,b}_\delta(\R\times\T)} \leq C\delta^{b'-b}\big( \|u\|_{X^{s,b}_\delta(\R\times\T)}^2 + \|v\|_{X^{s,b}_\delta(\R\times\T)}^2 \big)\|u-v\|_{X^{s,b}_\delta(\R\times\T)}.
\end{equation}
Accordingly, choose $C\delta^{b'-b} = \frac{1}{2(C\|u_0\|_{H^s(\R\times\T)}^2)}$, the contraction principle applying to prove local well-posedness in the space $X^{s,b}_\delta(\R\times\T)$ for data $u_0\in H^s(\R\times\T)$.

Next we prove global well-posedness. It suffices to show that the size $\delta$ of the time interval, on which   the local well-posedness result is known to hold, only dependents on a conserved quantity. For $s=1$, the conservation of energy yields a bound on the $H^1$-norm of $u(t)$, one has global well-posedness for $H^1$-data. For $s>1$, choose $s_0=\frac14$ in the second factor of \eqref{eq-3.9} so that one gets
\begin{equation}
\|u\|_{X^{s,b}_\delta(\R\times\T)} \leq C\|u_0\|_{H^s(\R\times\T)} + C\delta^{b'-b}\|u\|_{X^{1,b_3}_\delta(\R\times\T)}^2\|u\|_{X^{s,b_3}_\delta(\R\times\T)}
\end{equation}
for $u_0\in H^s$. Hence, a bound on $\|u\|_{X^{s,b}_\delta(\R\times\T)}$ for a time interval of size $\delta$ only dependents on $\|u_0\|_{H^1}$. This completes the proof of Theorem \ref{GWP}.

\section{growth~of~sobolev~norm~$\|u(t)\|_{H^s}$~with~$s>1$}
In this section, we mainly prove Theorem \ref{growth}. Without loss of generality, we just consider the case $t \to + \infty$. According to the section 3, the local well-posedness for the solution $u$ can be established on consecutive intervals $I=I_j=[t_j,t_{j+1}]$ which length $|I_j|=t_{j+1}-t_{j}$ is bounded by $\|u_0\|_{H^1}$.

The essential point is to establish an inequality
\begin{equation}\label{eq-4.1}
\|u(t_{j+1})\|_{H^s}^2 \leq \|u(t_{j})\|_{H^s}^2 + C\|u(t_{j})\|_{H^s}^{2-\gamma}
\end{equation}
for some fixed $\gamma = \gamma(s)>0$. From \eqref{eq-4.1} and by induction, we have
\begin{equation}
\|u(t)\|_{H^s(\R\times\T)}\leq C|t|^{\frac{1}{\gamma}}, \quad \text{for} \quad |t|\gg1,
\end{equation}
with $s>1$. One can refer to the details to Chen and the first author in \cite{CD}.

Now we will give the proof of \eqref{eq-4.1} under the condition that $s\in\N$ for simplicity. Denote by $\langle \cdot,\cdot\rangle$ the inner product in $L^2$ and $\dot{u}(t)=\partial_tu(t)$, then write
\begin{equation}
\begin{aligned}\label{eq-4.3}
\|u(t_{j+1})\|_{H^s}^2 - \|u(t_{j})\|_{H^s}^2 = & \int_{t_j}^{t_{j+1}}\big( \frac{d}{dt}\|u(t)\|_{H^s}^2 \big){\rm d}t\\
= & 2\Re\int_{t_j}^{t_{j+1}} \sum_{|\alpha_1|\leq s}\sum_{|\alpha_2|\leq s} \langle \partial^{\alpha_1}\dot{u}(t), \partial^{\alpha_2}u(t)\rangle{\rm d}t\\
= & 2\Im\int_{t_j}^{t_{j+1}} \sum_{|\alpha_1|\leq s}\sum_{|\alpha_2|\leq s} \langle \partial^{\alpha_1}(|u|^2u)(t), \partial^{\alpha_2}u(t)\rangle{\rm d}t.
\end{aligned}
\end{equation}
The contributions to \eqref{eq-4.3} are summarized as follows:
\begin{enumerate}
\item $I = \Im \int_{t_j}^{t_{j+1}}\int_{\R\times\T}|\sum_{|\alpha|\leq s}\partial^{\alpha}u(t)|^2|u(t)|^2{\rm d}x{\rm d}t$=0.
\item $II = \Im \int_{t_j}^{t_{j+1}}\int_{\R\times\T}\big( \sum_{|\alpha|\leq s}\partial^{\alpha}\bar{u}(t) \big)^2(u(t))^2{\rm d}x{\rm d}t$.
\item
\begin{equation*}
\begin{aligned}
III = & \Im \int_{t_j}^{t_{j+1}}\int_{\R\times\T} \sum_{|\alpha|\leq s}\big(\partial^{\alpha}\bar{u}(t)\big)\\
& \sum_{|\alpha_1| + |\alpha_2| + |\alpha_3| \leq s} \big(\partial^{\alpha_1}u_1(t)\partial^{\alpha_2}u_2(t)\partial^{\alpha_3}u_3(t)\big){\rm d}x{\rm d}t,
\end{aligned}
\end{equation*}
here the partial differential is with respect to the space variable $x$, one of the $u_i = u$ or $\bar{u}$, and $|\alpha_i|<s$, $\sum_{i=1}^3|\alpha_i|=s$, and at most one $|\alpha_i| = 0$.
\end{enumerate}

\subsection{The~contribution~of~$(3)$}
Without loss of generality, we suppose $|\alpha_1|\geq|\alpha_2|\geq|\alpha_3|$. Therefore, $(3)$ can be divided into two cases:
\begin{enumerate}
\item[(a)] $|\alpha_3|>0$.
\item[(b)] $|\alpha_1|\geq|\alpha_2|>0$ and $|\alpha_3|=0$.
\end{enumerate}

Hence, $III\leq III_1+III_2$, $III_1$ corresponds to $(a)$, and $III_2$ corresponds to $(b)$. We first consider the case $III_1$. H\"older's inequality and Theorem \ref{theorem-nonlinear estimate} yield
\begin{equation}
\begin{aligned}\label{eq-4.4'}
III_1 \lesssim & \sum_{|\alpha|\leq s}\big\| \partial^{\alpha}\bar{u} \big\|_{X^{-\varepsilon_1,1-b'}_{I_j}}\sum_{|\alpha_1| + |\alpha_2| + |\alpha_3| \leq s}\big\| \partial^{\alpha_1}u_1\partial^{\alpha_2}u_2\partial^{\alpha_3}u_3 \big\|_{X^{\varepsilon_1,b'-1}_{I_j}}\\
\lesssim & \|u\|_{X^{s-\varepsilon_1,1-b'}_{I_j}}\sum_{|\alpha_1| + |\alpha_2| + |\alpha_3| \leq s}\prod_{i=1}^3\big\| \partial^{\alpha_i}u_i \big\|_{X^{\varepsilon_1,b}_{I_j}}\\
\lesssim & \|u\|_{X^{s-\varepsilon_1,b}_{I_j}}\prod_{i=1}^3\|u_i\|_{X^{|\alpha_i|+\varepsilon_1,b}_{I_j}},
\end{aligned}
\end{equation}
where $\varepsilon_1 >0$ and $b'\geq b>\frac12$. For $1 \leq s-\varepsilon_1 < s$, by interpolation
\begin{equation}
\begin{aligned}\label{eq-4.4}
\|u\|_{X^{s-\varepsilon_1,b}_{I_j}} \lesssim & \|u\|_{X^{s,b}_{I_j}}^{1-\frac{\varepsilon_1}{s-1}}\|u\|_{X^{1,b}_{I_j}}^{\frac{\varepsilon_1}{s-1}} \lesssim  \|u(t_j)\|_{H^s}^{1-\frac{\varepsilon_1}{s-1}}\|u(t_j)\|_{H^1}^{\frac{\varepsilon_1}{s-1}}\\
\lesssim & \|u(t_j)\|_{H^s}^{1-\frac{\varepsilon_1}{s-1}}.
\end{aligned}
\end{equation}
Similarly, since $\varepsilon_1 < 1$, $0<|\alpha_i|+\varepsilon_1<s$, then
\begin{equation}
\begin{aligned}\label{eq-4.5}
\|u_i\|_{X^{|\alpha_i|+\varepsilon_1,b}_{I_j}} \lesssim & \|u_i\|_{X^{s,b}_{I_j}}^{\frac{|\alpha_i|+\varepsilon_1}{s}}\|u_i\|_{X^{0,b}_{I_j}}^{1-\frac{|\alpha_i|+\varepsilon_1}{s}} \\ \lesssim & \|u(t_j)\|_{H^s}^{\frac{|\alpha_i|+\varepsilon_1}{s}}.
\end{aligned}
\end{equation}
Plugging \eqref{eq-4.4} and \eqref{eq-4.5} into \eqref{eq-4.4'} yields
\begin{equation}\label{eq-4.7}
III_1 \lesssim \|u(t_j)\|_{H^s}^{1-\frac{\varepsilon_1}{s-1} + \frac{\sum_{i=1}^3|\alpha_i|+3\varepsilon_1-3}{s-1}} = \|u(t_j)\|_{H^s}^{2-2\frac{1-\varepsilon_1}{s-1}}.
\end{equation}

Similar to case $(a)$, we obtain
\begin{equation}
\begin{aligned}\label{eq-4.8}
III_2 \lesssim & \|u\|_{X^{s-\varepsilon_1,b}_{I_j}}\prod_{i=1}^2\|u_i\|_{X^{|\alpha_i|+\varepsilon_1,b}_{I_j}}\|u_3\|_{X^{\varepsilon,b}}\\
\lesssim & \|u(t_j)\|_{H^s}^{1-\frac{\varepsilon_1}{s-1}+\frac{\sum_{i=1}^2|\alpha_i|+2\varepsilon_1}{s-1}}
= \|u(t_j)\|_{H^s}^{2-\frac{1-\varepsilon_1}{s-1}}.
\end{aligned}
\end{equation}
Therefore, \eqref{eq-4.7} and \eqref{eq-4.8} implies
\begin{equation}
III\leq III_1 + III_2 \lesssim \|u(t_j)\|_{H^s}^{2-\frac{1-\varepsilon_1}{s-1}}.
\end{equation}

\subsection{The~contribution~of~$(2)$}

From the H\"older inequality, we obtain
\begin{equation}\label{eq-4.14}
II \lesssim \|\sum_{|\alpha|\leq s}\partial^\alpha\bar{u}\|_{X^{0,b}_{I_j}}\|\sum_{|\alpha|\leq s}\partial^\alpha\bar{u}(u)^2\|_{X^{0,-b}_{I_j}} \lesssim \|u(t_j)\|_{H^s}\|\sum_{|\alpha|\leq s}\partial^\alpha\bar{u}(u)^2\|_{X^{0,-b}_{I_j}}.
\end{equation}
In order to estimate $\|\sum_{|\alpha|\leq s}\partial^\alpha\bar{u}(u)^2\|_{X^{0,-b}_{I_j}}$, we consider it by duality
\begin{equation}
\begin{aligned}\label{eq-5.2}
\big\|& \sum_{|\alpha|\leq s}\partial^\alpha\bar{u}(u)^2\big\|_{X^{0,-b}_{I_j}} = \sup_{\|\hat{w}\|_{L^2}=1}\sum\int\int\frac{\bar{\hat{w}}(\tau,\xi,n)}{(1+|\tau-\xi^2-n^2|)^b}(1+|\xi|+|n|)^s\\
& \times |\bar{\hat{u}}(-\tau_1,-\xi_1,-n_1)||\hat{u}(\tau_2,\xi_2,n_2)||\hat{u}(\tau_3,\xi_3,n_3)|{\rm d}\tau_1{\rm d}\xi_1{\rm d}\tau_2{\rm d}\xi_2{\rm d}\tau_3{\rm d}\xi_3,
\end{aligned}
\end{equation}
where $n = n_1+n_2+n_3$, $\xi = \xi_1+\xi_2+\xi_3$ and $\tau = \tau_1+\tau_2+\tau_3$.
\begin{enumerate}
  \item[$\mathbf{Case~1}$.]
       Suppose $(1+|\xi_2|+|n_2|) \geq C(1+|\xi_1|+|n_1|)^{\theta}$ or $(1+|\xi_3|+|n_3|) \geq C(1+|\xi_1|+|n_1|)^{\theta}$, where $0<\theta<1$ is a small parameter and $C$ that will be fixed later.

        Without loss of generality, take $(1+|\xi_3|+|n_3|) \geq C(1+|\xi_1|+|n_1|)^{\theta}$ for example, and the another case could be treated in the same way. Similar to estimate \eqref{eq-3.3}, we also can estimate \eqref{eq-5.2} by
      \begin{equation}
      \eqref{eq-5.2} \lesssim \|u\|_{X^{s,b_3}_{I_j}}\|u\|_{X^{s_0,b_3}_{I_j}}^2
      \end{equation}
      for any $0<s_0<\frac14$ and $b_3<\frac12<b$. From $(1+|\xi_3|+|n_3|)^{s_0-1} \leq  C^{s_0-1}(1+|\xi_1|+|n_1|)^{\theta(s_0-1)}$, the formula \eqref{eq-3.8}, and $\|u\|_{X^{1,b}} \lesssim \|u_0\|_{H^1}$, we obtain
      \begin{equation}
      \begin{aligned}\label{eq-4.13}
      \eqref{eq-5.2} \lesssim & \|u\|_{X^{s-\theta(1-s_0),b_3}_{I_j}}\|u\|_{X^{s_0,b_3}_{I_j}}\|u\|_{X^{1,b_3}_{I_j}} \lesssim \|u\|_{X^{s-\theta(1-s_0),b}_{I_j}}\|u\|_{X^{1,b}_{I_j}}^2\\
      \lesssim & \|u\|_{X^{s,b}_{I_j}}^{1-\frac{\theta(1-s_0)}{s-1}} \lesssim \|u(t_j)\|_{H^s}^{1-\frac{\theta(1-s_0)}{s-1}}.
      \end{aligned}
      \end{equation}

  \item[$\mathbf{Case~2}$.] Suppose $(1+|\xi_2|+|n_2|) < C(1+|\xi_1|+|n_1|)^{\theta}$ and $(1+|\xi_3|+|n_3|) < C(1+|\xi_1|+|n_1|)^{\theta}$.

  \item[$\mathbf{Subcase ~1}$.] Suppose $|\tau_2| \geq C_1(1+|\xi_1|+|n_1|)^2$ or $|\tau_3| \geq C_1(1+|\xi_1|+|n_1|)^2$, where $C_1$ will be fixed later.

  Without loss of generality, take $|\tau_3| \geq C_1(1+|\xi_1|+|n_1|)^2$ for example. Since
  \begin{equation}
  \begin{aligned}\label{eq-5.5}
  1+|\tau_3-\xi_3^2-n_3^2| \geq & |\tau_3| - (1+|\xi_3|+|n_3|)^2 \\
  \geq & C_1(1+|\xi_1|+|n_1|)^2 - C(1+|\xi_1|+|n_1|)^{2\theta}\\
  = & C_1(1+|\xi_1|+|n_1|)^2\big( 1-\frac{C}{C_1}(1+|\xi_1|+|n_1|)^{2\theta-2} \big)\\
  \geq & C'(1+|\xi_1|+|n_1|)^2.
  \end{aligned}
  \end{equation}
  In particular, if $|\xi_1| = |n_1| = 0$, choose $C<C_1$. Hence, \eqref{eq-5.5} implies $\frac{1}{(1+|\tau_3-\xi_3^2-n_3^2|)^{b-\frac14}}\leq \frac{1}{(C'(1+|\xi_1|+|n_1|)^2)^{b-\frac14}}$. By the formula \eqref{eq-3.8}, we get
  \begin{equation}
      \begin{aligned}\label{eq-4.15}
      \eqref{eq-5.2} \lesssim & \|u\|_{X^{s-2(b-\frac14),b_3}_{I_j}}\|u\|_{X^{s_0,b_3}_{I_j}}\|u\|_{X^{s_0,b_3+(b-\frac14)}_{I_j}} \lesssim \|u\|_{X^{s-2(b-\frac14),b}_{I_j}}\|u\|_{X^{1,b}_{I_j}}^2\\
      \lesssim & \|u\|_{X^{s,b}_{I_j}}^{1-\frac{2(b-\frac14)}{s-1}}
      \lesssim  \|u(t_j)\|_{H^s}^{1-\frac{2(b-\frac14)}{s-1}}.
      \end{aligned}
      \end{equation}

  \item[$\mathbf{Subcase ~2}$.]  Now we just need to consider the case of $|\tau_2| \geq C_1(1+|\xi_1|+|n_1|)^2$ and $|\tau_3| \geq C_1(1+|\xi_1|+|n_1|)^2$.
       Denote
       $$
       (1+|\xi_1|+|n_1|)^s(1+|\tau_1+\xi_1^2+n_1^2|)^{b}|\bar{\hat{u}}(-\tau_1,-\xi_1,-n_1)| = \tilde{u}(-\tau_1,-\xi_1,-n_1),
       $$
       Rewrite \eqref{eq-5.2} as
       \begin{equation}
       \begin{aligned}\label{eq-5.7}
       & \sup_{\|\hat{w}\|_{L^2}=1}\sum\int\int\frac{|\bar{\hat{w}}(\tau,\xi,n)|}{(1+|\tau-\xi^2-n^2|)^b}
       \frac{|\tilde{u}(-\tau_1,-\xi_1,-n_1)|}{(1+|\tau_1+\xi_1^2+n_1^2|)^b}\\
       & \qquad\qquad\times |\hat{u}(\tau_2,\xi_2,n_2)||\hat{u}(\tau_3,\xi_3,n_3)|{\rm d}\tau_1{\rm d}\xi_1{\rm d}\tau_2{\rm d}\xi_2{\rm d}\tau_3{\rm d}\xi_3,
       \end{aligned}
       \end{equation}
       where $\tau = \tau_1+\tau_2+\tau_3, \xi = \xi_1+\xi_2+\xi_3$ and $n = n_1+n_2+n_3$. Since
       \begin{equation}
       \begin{aligned}\label{eq-5.8}
       1+|\tau-\xi^2-n^2| + 1+|\tau_1+\xi_1^2+n_1^2| \geq & 1+|\xi_1^2+n_1^2|+1+|\xi^2+n^2|-|\tau-\tau_1|\\
       \geq & (1+|\xi_1|+|n_1|)^2 - 2C_1(1+|\xi_1|+|n_1|)^2\\
       \geq & C''(1+|\xi_1|+|n_1|)^2
       \end{aligned}
       \end{equation}
       for $C_1 <\frac12$.
       And rewrite \eqref{eq-5.7} as
       \begin{equation}
       \begin{aligned}\label{eq-5.9}
       & \sup_{\|\hat{w}\|_{L^2}=1}\sum\int\int\frac{|\bar{\hat{w}}(\tau,\xi,n)|}
       {(1+|\tau-\xi^2-n^2|)^{\frac14}}(1+|\tau-\xi^2-n^2|)^{\frac14-b}
       \frac{|\tilde{u}(-\tau_1,-\xi_1,-n_1)|}{(1+|\tau_1+\xi_1^2+n_1^2|)^{\frac14}}\\
       & \qquad\qquad \times (1+|\tau_1+\xi_1^2+n_1^2|)^{\frac14-b}|\hat{u}(\tau_2,\xi_2,n_2)||\hat{u}(\tau_3,\xi_3,n_3)|{\rm d}\tau_1{\rm d}\xi_1{\rm d}\tau_2{\rm d}\xi_2{\rm d}\tau_3{\rm d}\xi_3,
       \end{aligned}
       \end{equation}
       where $\tau = \tau_1+\tau_2+\tau_3, \xi = \xi_1+\xi_2+\xi_3$ and $n = n_1+n_2+n_3$.
       Using \eqref{eq-5.8}, we have
       $$
       (1+|\tau-\xi^2-n^2|)^{\frac14-b}(1+|\tau_1+\xi_1^2+n_1^2|)^{\frac14-b} \lesssim (1+|\xi_1|+|n_1|)^{2(\frac14-b)}
       $$
       for $b_3<\frac14<\frac12<b$. This yields
       \begin{equation}
      \begin{aligned}\label{eq-4.19}
      \eqref{eq-5.9} \lesssim & \|u\|_{X^{s-2(b-\frac14),b}_{I_j}}\|u\|_{X^{s_0,b_3}_{I_j}}^2 \lesssim \|u\|_{X^{s-2(b-\frac14),b}_{I_j}}\|u\|_{X^{1,b}_{I_j}}^2\\
      \lesssim & \|u\|_{X^{s,b}_{I_j}}^{1-\frac{2(b-\frac14)}{s-1}}
      \lesssim  \|u(t_j)\|_{H^s}^{1-\frac{2(b-\frac14)}{s-1}}.
      \end{aligned}
      \end{equation}
\end{enumerate}
Combining \eqref{eq-4.13}, \eqref{eq-4.19} and \eqref{eq-4.15}, \eqref{eq-4.14} is reduced to
\begin{equation}\label{eq-4.20}
II \lesssim \|u(t_j)\|_{H^s}\|u(t_j)\|_{H^s}^{1-\frac{\eta}{s-1}},
\end{equation}
where $\eta$ satisfying
\begin{equation*}
\eta = \min\{ \theta(1-s_0), 2(b-\frac14) \}.
\end{equation*}
To equivalent this two terms, we have
$
\theta = \frac{2(b-\frac14)}{1-s_0}<1
$
for $0<s_0<\frac14$ and $\frac14<\frac12<b$.

Comparing \eqref{eq-4.20} with the result of $III$, there is $\gamma=\frac12-$ such that
\begin{equation}
\eqref{eq-4.3} \lesssim \|u(t_j)\|_{H^s}\|u(t_j)\|_{H^s}^{1-\frac{\gamma}{s-1}}.
\end{equation}
Therefore, we conclude the proof of the claim \eqref{eq-4.1}.

\noindent \textbf{Acknowledgments.}
We highly appreciate Prof. Jiqiang Zheng and Prof. Qionglei Chen for helpful discussions and Fanfei Meng, Ying Wang for beneficial suggestions on this paper.

K. Yang was supported by a Doctoral Fundation of Chongqing Normal University and Chongqing Science and Technology Commission (21XLB025,ncamc2022-msxm04),  and a funding(6142A0521Q06,HX02021-36)from Laboratory  of  Computational  Physics, Institute of Applied Physics and Computational Mathematics in Beijing.

\end{document}